\documentclass[a4paper,12pt]{article}
\usepackage{amsmath}
\usepackage{amsthm}
\usepackage{amssymb}
\usepackage{enumerate}
\usepackage{url}
\usepackage[utf8]{inputenc}

\title{Invariance of three-dimensional Bessel bridges in terms of time reversal}

\author{Yuu Hariya\thanks{Supported in part by JSPS KAKENHI Grant Number  22K03330}}

\date{\empty}
\numberwithin{equation}{section}

\theoremstyle{plain}

\newtheorem{thm}{Theorem}[section]
\newtheorem{lem}[thm]{Lemma}
\newtheorem{prop}[thm]{Proposition}
\newtheorem{cor}[thm]{Corollary}

\theoremstyle{definition}

\theoremstyle{remark}
\newtheorem{rem}[thm]{Remark}

\begin{document}

\newcommand\ND{\newcommand}
\newcommand\RD{\renewcommand}

\ND\N{\mathbb{N}}
\ND\R{\mathbb{R}}
\ND\C{\mathbb{C}}

\ND\F{\mathcal{F}}

\ND\kp{\kappa}

\ND\ind{\boldsymbol{1}}

\ND\al{\alpha }
\ND\la{\lambda }
\ND\La{\Lambda }
\ND\ve{\varepsilon}
\ND\Om{\Omega}

\ND\ga{\gamma}

\ND\lref[1]{Lemma~\ref{#1}}
\ND\tref[1]{Theorem~\ref{#1}}
\ND\pref[1]{Proposition~\ref{#1}}
\ND\sref[1]{Section~\ref{#1}}
\ND\ssref[1]{Subsection~\ref{#1}}
\ND\aref[1]{Appendix~\ref{#1}}
\ND\rref[1]{Remark~\ref{#1}} 
\ND\cref[1]{Corollary~\ref{#1}}
\ND\csref[1]{Corollaries~\ref{#1}}
\ND\eref[1]{Example~\ref{#1}}
\ND\fref[1]{Fig.\ {#1} }
\ND\lsref[1]{Lemmas~\ref{#1}}
\ND\tsref[1]{Theorems~\ref{#1}}
\ND\dref[1]{Definition~\ref{#1}}
\ND\psref[1]{Propositions~\ref{#1}}
\ND\rsref[1]{Remarks~\ref{#1}}
\ND\sssref[1]{Subsections~\ref{#1}}

\ND\pr{\mathbb{P}}
\ND\ex{\mathbb{E}}

\ND\Ga{\Gamma}

\ND\eqd{\stackrel{(d)}{=}}

\ND\cm{\mathcal{M}}

\ND\cmb{\overline{\cm}}

\ND\cp{\mathcal{P}}

\ND\cpb{\overline{\cp}}

\ND\rd{\mathring{\rho}}
\ND\Rd{\mathring{R}}
\ND\rphi{\mathring{\phi}}
\ND\sg{\sigma}
\ND\ctd[2]{C([0,#1];\mathbb{R}^{#2})}
\ND\ctx{\mathcal{T}^{c}_{x}}

\ND\bes{r}
\ND\pbes[2]{\mathbb{P}^{(3)}_{#1,#2}}
\ND\ebes[2]{\mathbb{E}^{(3)}_{#1,#2}}
\ND\pBES[1]{\mathbb{P}^{(3)}_{#1}}
\ND\eBES[1]{\mathbb{E}^{(3)}_{#1}}

\ND\bb{\beta}
\ND\pbb[2]{\mathbb{P}_{#1,#2}}
\ND\ebb[2]{\mathbb{E}_{#1,#2}}
\ND\bbd{\mathring{\beta}}

\ND\cn{\mathcal{N}}
\ND\cq{\mathcal{Q}}
\ND\cs{\mathcal{S}}

\ND\id{\mathrm{Id}}

\ND{\rmid}[1]{\mathrel{}\middle#1\mathrel{}}

\def\thefootnote{{}}

\maketitle 
\begin{abstract}
Given $a,b\ge 0$ and $t>0$, let $\rho =\{ \rho _{s}\} _{0\le s\le t}$ be a three-dimensional 
Bessel bridge from $a$ to $b$ over $[0,t]$. In this paper, based on a conditional identity in law 
between Brownian bridges stemming from Pitman's theorem, we show in particular 
that the process given by 
\begin{align*}
 \rho _{s}+\Bigl| b-a+
 \min _{0\le u\le s}\rho _{u}-\min _{s\le u\le t}\rho _{u}
 \Bigr| 
 -\Bigl| 
 \min _{0\le u\le s}\rho _{u}-\min _{s\le u\le t}\rho _{u}
 \Bigr| ,\quad 0\le s\le t,
\end{align*}
has the same law as the time reversal $\{ \rho _{t-s}\} _{0\le s\le t}$ of $\rho $. 
As an immediate application, letting $R=\{ R_{s}\} _{s\ge 0}$ be a three-dimensional 
Bessel process starting from $a$, we obtain the following time-reversal and 
time-inversion results on $R$: $\{ R_{t-s}\} _{0\le s\le t}$ is identical in law with 
the process given by 
\begin{align*}
 R_{s}+R_{t}-2\min _{s\le u\le t}R_{u},\quad 0\le s\le t,
\end{align*}
when $a=0$, and $\{ sR_{1/s}\} _{s>0}$ is identical in law with the process given by 
\begin{align*}
 R_{s}-2(1+s)\min _{0\le u\le s}\frac{R_{u}}{1+u}+a(1+s),\quad s>0,
\end{align*}
for every $a\ge 0$.
\footnote{Mathematical Institute, Tohoku University, Aoba-ku, Sendai 980-8578, Japan}
\footnote{E-mail: hariya@tohoku.ac.jp}
\footnote{{\itshape Keywords and Phrases}:~{Brownian motion}; {three-dimensional Bessel bridge}; {Pitman's transformation}}
\footnote{{\itshape MSC 2020 Subject Classifications}:~Primary~{60J65}; Secondary~{60J60}}
\end{abstract}

\section{Introduction and main results}\label{;intro}

Three-dimensional Bessel bridges are fundamental stochastic 
processes in probability theory. When their length is $1$, 
the special case in which the process is null at both ends of 
the time interval $[0,1]$ is referred to as a standard three-dimensional 
Bessel bridge, which is known to be identical in law with a normalized 
Brownian excursion \cite{by}. As is well known, 
a Brownian excursion is a key component in the excursion theory of 
Brownian motion, and deep studies have been conducted such as 
its path decomposition, its local time process, and so on; see 
\cite{bp, by, jeu, my, ry} and references therein. On the other hand, 
it seems that less is known in the general case of arbitrary starting 
and end points. An indication of subtlety in such a case is that, unlike 
a standard three-dimensional Bessel bridge, the process cannot be 
realized as the modulus of a three-dimensional Brownian bridge unless 
either starting or end point coincides with the origin \cite{yz}. 
We expect that identities in law obtained in this paper will be an asset 
to a better understanding of that subtle case. 

In the course of our exploration, Pitman's $2M-X$ theorem \cite{jwp} as well 
as Pitman's transformation play a significant role. Pitman's transformation 
and it generalizations have a connection with a path model called the 
Littlemann path model in representation theory and have been used in 
the study of eigenvalues of random matrices of a certain class; see, e.g., 
\cite{bbo, bj}. In recent studies of classical integrable systems such as 
Toda lattice and box-ball systems from probabilistic perspective, there 
is also a renewed interest in Pitman's and Pitman-type transformations 
\cite{ckst, cs1, cs2, cst1, cst2}; our results include a number of novel 
identities involving Pitman's transformation along with new findings on 
its property, which we also expect will provide a deeper insight on that 
important transformation.

Given $t>0$, let $C([0,t];\R )$ be the space of real-valued continuous functions over 
$[0,t]$. Unless otherwise specified, $t$ is fixed throughout the paper. For 
each $x\in \R $, define $\cm _{x}: C([0,t];\R )\to C([0,t];\R )$ by  
\begin{equation}\label{;cmx}
\begin{split}
 &\cm _{x}(\phi )(s)\\
 &:=\phi _{s}-\frac{\phi _{t}-x}{2}-\left| 
 \frac{\phi _{t}-x}{2}+\max _{0\le u\le s}\phi _{u}-\max _{s\le u\le t}\phi _{u}
 \right| 
 +\Bigl| 
 \max _{0\le u\le s}\phi _{u}-\max _{s\le u\le t}\phi _{u}
 \Bigr| 
\end{split}
\end{equation}
for $0\le s\le t$ and $\phi \in C([0,t];\R )$. We also set 
\begin{align*}
 \cmb _{x}(\phi ):=-\cm _{-x}(-\phi ),\quad \phi \in C([0,t];\R ),\ x\in \R ,
\end{align*}
that is, 
\begin{equation}\label{;cmbx}
\begin{split}
 &\cmb _{x}(\phi )(s)\\
 &:=\phi _{s}-\frac{\phi _{t}-x}{2}+\left| 
 \frac{\phi _{t}-x}{2}+\min _{0\le u\le s}\phi _{u}-\min _{s\le u\le t}\phi _{u}
 \right| 
 -\Bigl| 
 \min _{0\le u\le s}\phi _{u}-\min _{s\le u\le t}\phi _{u}
 \Bigr| 
\end{split}
\end{equation}
for $0\le s\le t$. Recall the well-known Pitman's 
transformation, for which we write $\cp $, on the space $C([0,\infty );\R )$ 
of real-valued continuous functions $\phi $ over $[0,\infty )$, defined by 
\begin{align}
 \cp (\phi )(s)&:=2\max _{0\le u\le s}\phi _{u}-\phi _{s},\quad s\ge 0, \label{;cp}
\intertext{with which we also define $\cpb (\phi ):=\cp (-\phi )$, namely} 
 \cpb (\phi )(s)&:=\phi _{s}-2\min _{0\le u\le s}\phi _{u},
 \quad s\ge 0. \label{;cpb}
\end{align}
We keep the same notation if we restrict $\cp $ and $\cpb $ to $C([0,t];\R )$. 
For every $\phi \in C([0,t];\R )$, we also write $\rphi $ for 
its time reversal:
\begin{align}\label{;ts}
 \rphi _{s}:=\phi _{t-s},\quad 0\le s\le t. 
\end{align}

Given $a,b\ge 0$, let $\rho =\{ \rho _{s}\} _{0\le s\le t}$ be a three-dimensional 
Bessel bridge from $a$ to $b$ over $[0,t]$; for the definition of Bessel bridges of 
positive dimension, see \cite[Chapter~XI, Section~3]{ry}. One of the main results of 
this paper is stated as follows: 
\begin{thm}\label{;tmain1}
For every $a,b\ge 0$, the three-dimensional process 
\begin{align*}
\left( 
\cmb _{b-a}(\rd -b)(s)+a,\,\cpb (\rd )(s)+b,\,\rd _{s}
\right) ,\quad 0\le s\le t,
\end{align*}
has the same law as 
\begin{align*}
\left( 
\rho _{s},\,\cpb (\rho )(s)+a,\,\cmb _{a-b}(\rho -a)(s)+b
\right) ,\quad 0\le s\le t.
\end{align*}
\end{thm}

As will be seen in \sref{;prftmain1}, \tref{;tmain1} is obtained as 
a particular case of \tref{;genr}. We remark that, if we introduce 
two more transformations $\cn $ and $\cq $ defined respectively by 
\begin{align}
 \cn (\phi )(s):=\cmb _{\phi _{0}-\phi _{t}}(\phi -\phi _{0})(s)+\phi _{t}, 
 \quad &0\le s\le t,\ \phi \in \ctd{t}{},
 \label{;cn} \\
 \cq (\phi )(s):=\cpb (\phi )(s)+\phi _{0},\quad &s\ge 0,\ \phi \in C([0,\infty );\R ), \label{;cq}
\end{align}
then the assertion of the theorem is stated more concisely as
\begin{align}\label{;concise}
 \left( 
 \cn (\rd ),\,\cq (\rd ),\,\rd 
 \right) \eqd 
 \left( 
 \rho ,\,\cq (\rho ),\,\cn (\rho )
 \right) .
\end{align}
Here the equality stands for the identity in law; we also regard $\cq$ as 
a transformation on $\ctd{t}{}$. By the definition~\eqref{;cmbx} of the transformations 
$\cmb _{x}$, the transformation $\cn $ admits the expression 
\begin{align}\label{;cnexpr}
 \cn (\phi )(s)=
 \phi _{s}+\Bigl| \phi _{t}-\phi _{0}+
 \min _{0\le u\le s}\phi _{u}-\min _{s\le u\le t}\phi _{u}
 \Bigr| 
 -\Bigl| 
 \min _{0\le u\le s}\phi _{u}-\min _{s\le u\le t}\phi _{u}
 \Bigr| 
\end{align}
for $0\le s\le t$ and $\phi \in \ctd{t}{}$. Among other properties 
investigated in \ssref{;sspnq}, the transformation $\cn $ satisfies 
\begin{align*}
 \cn (\phi )(0)=\phi _{t} && \text{and} && \cn (\phi )(t)=\phi _{0}
\end{align*}
for all $\phi \in \ctd{t}{}$. The above property is consistent with 
the theorem in the sense that \eqref{;concise} entails 
\begin{align*}
 \pr \!\left( 
 \cn (\rd )(0)=a,\,\cn (\rd )(t)=b
 \right) =1.
\end{align*}
Furthermore, the theorem indicates that $\cn $ is an 
involution, preserves $\cq $, and commutes with the 
operation~\eqref{;ts} of time reversal:   
\begin{align*}
 \cn \circ \cn =\id , && \cq \circ \cn =\cq , && 
 \overbrace{\cn (\phi )}^{\circ }=\cn \bigl( \rphi \bigr) 
 \quad (\forall \phi \in \ctd{t}{}),
\end{align*}
which is indeed the case as a consequence of \lref{;mxpl}; see \ssref{;sspnq}. 
Here $\id $ is the identity map on $\ctd{t}{}$. 
When $a=b$, expression~\eqref{;cnexpr} entails that the theorem 
reduces to the well-known fact that $\rd \eqd \rho $ (see, e.g, 
\cite[Chapter~XI, Exercise~\thetag{3.7}]{ry}).

Let $R=\{ R_{s}\} _{s\ge 0}$ be a three-dimensional 
Bessel process starting from $a$. As an immediate consequence 
of \tref{;tmain1}, we have the next two \csref{;cor1} and \ref{;cor2}.

\begin{cor}\label{;cor1}
For every $a\ge 0$, it holds that 
\begin{align}\label{;cor1q1}
 \Bigl( 
 \cn (\Rd ),\,\cq \bigl( \Rd \bigr) ,\,\Rd 
 \Bigr) 
 \eqd \left( 
 R,\,\cq (R),\,\cn (R)
 \right) .
\end{align}
In particular, when $a=0$, 
\begin{equation}\label{;cor1q2}
\begin{split}
&\Bigl\{ 
\Bigl( 
R_{t-s}+R_{t}-2\min _{t-s\le u\le t}R_{u},\,R_{t-s}
\Bigr) 
\Bigr\} _{0\le s\le t}\\
&\eqd 
\Bigl\{ 
\Bigl( R_{s},\,
R_{s}+R_{t}-2\min _{s\le u\le t}R_{u}
\Bigr) 
\Bigr\} _{0\le s\le t}.
\end{split}
\end{equation}
\end{cor}

\begin{cor}\label{;cor2}
Define the process $X=\{ X_{s}\} _{s\ge 0}$ by 
\begin{align*}
 X_{s}:=\frac{R_{s}}{1+s},\quad s\ge 0.
\end{align*}
Then it holds that 
\begin{align*}
 \left\{ 
 \left( 
 \cq (X)(1/s),\,X_{1/s}
 \right) 
 \right\} _{s>0}
 \eqd 
 \left\{ 
 \left( 
 X_{s},\,\cq (X)(s)
 \right) 
 \right\} _{s>0},
\end{align*}
which, by multiplying both sides by the deterministic function 
$1+s,\,s>0$, is equivalently stated as 
\begin{align*}
&\left\{ 
\left( 
sR _{1/s}-2(1+s)\min _{0\le u\le 1/s}\frac{R_{u}}{1+u}+a(1+s),\,sR_{1/s}
\right) 
\right\} _{s>0}\\
&\eqd 
\left\{ 
\left( 
R_{s},\,R_{s}-2(1+s)\min _{0\le u\le s}\frac{R_{u}}{1+u}+a(1+s)
\right) 
\right\} _{s>0}.
\end{align*}
\end{cor}

We refer the reader to \ssref{;prfcors12} for counterparts to the above two corollaries 
in the case of Brownian motion; see \eqref{;btsd} and \tref{;binvthm}.

Notice that the first component in the left-hand side of \eqref{;cor1q2}, as well as 
the second component in the right-hand side of \eqref{;cor1q2}, gives a 
nonnegative process as $R$ does:
\begin{align*}
 R_{t-s}+R_{t}-2\min _{t-s\le u\le t}R_{u}
 &\ge R_{t-s}+R_{t}-2
 \min \left\{ 
 R_{t-s},R_{t}
 \right\} \\
 &=\left| 
 R_{t-s}-R_{t}
 \right| \ge 0
\end{align*}
for $0\le s\le t$. 
\cref{;cor1} is obtained by letting the parameter $b$ in \tref{;tmain1} be distributed 
according to the law of $R_{t}$. On the other hand, if we consider the case 
$t=1$ and $b=0$ in the theorem, then, thanks to the identity in law 
\begin{align}\label{;inver}
 \{ \rho _{s}\} _{0\le s<1}\eqd 
 \left\{ 
 (1-s)R_{s/(1-s)}
 \right\} _{0\le s<1}
\end{align}
in that case (see, e.g., \cite[Chapter~XI, Exercise~\thetag{3.6}]{ry}), 
we obtain \cref{;cor2}; note that, when $a=0$, the corollary reduces to 
the usual time-inversion of the standard three-dimensional Bessel process: 
$\{ sR_{1/s}\} _{s>0}\eqd \{ R_{s}\} _{s>0}$, which is because 
\begin{align*}
 \pr \!\left( \min _{0\le u\le s}\frac{R_{u}}{1+u}=0\ \text{for all $s\ge 0$}\right) =1.
\end{align*}
See \ssref{;prfcors12} for more details of the proofs of the above two corollaries.

For every $x\in \R $, let 
$\bb ^{x}=\{ \bb ^{x}_{s}\} _{0\le s \le t}$ denote a one-dimensional Brownian bridge 
from $0$ to $x$ over $[0,t]$. Notice that these 
Brownian bridges are related via 
\begin{align}\label{;bb}
 \bb ^{x}\eqd \left\{ \bb ^{y}_{s}+\frac{x-y}{t}s\right\} _{0\le s \le t}
\end{align}
for any $x,y\in \R $.
The proof of our \tref{;tmain1} hinges upon the following conditional 
identity between Brownian bridges, which is new to our best knowledge and 
will be used to derive \tref{;genr}, a generalization of \tref{;tmain1}.

\begin{thm}\label{;tmain2}
Let $x,y\in \R $ be such that $|x|\ge |y|$. Then the three-dimensional process 
\begin{align*}
 \left( 
 \bb ^{x}_{s},\,\cp (\bb ^{x})(s),\,\cm _{y}(\bb ^{x})(s)
 \right) ,\quad 0\le s\le t,
\end{align*}
is identical in law with the process 
\begin{align*}
 \left( 
 \cm _{x}(\bb ^{y})(s),\,\cp (\bb ^{y})(s),\,\bb ^{y}_{s}
 \right) ,\quad 0\le s\le t,
\end{align*}
conditioned on the event that $\cp (\bb ^{y})(t)\ge |x|$.
\end{thm}

The above theorem indicates in particular that, a.s.,   
\begin{align}
 \cm _{y}(\bb ^{x})(0)=0 && \text{and} && \cm _{y}(\bb ^{x})(t)=y \label{;bv1}
\intertext{when $|x|\ge |y|$, and that, a.s.\ on the event that $\cp (\bb ^{y})(t)\ge |x|$,}
 \cm _{x}(\bb ^{y})(0)=0 && \text{and} && \cm _{x}(\bb ^{y})(t)=x. \label{;bv2}
\end{align}
In view of \lref{;mxbv} below, these are 
really the case. As for \eqref{;bv1}, note that
\begin{align}\label{;plb}
 \pr \!\left( 
 \cp (\bb ^{x})(t)\ge |x|
 \right) =1; 
\end{align}
indeed, by the definition of $\cp $, 
\begin{align*}
 \cp (\bb ^{x})(t)\ge 2\max \left\{ \bb ^{x}_{0},\bb ^{x}_{t}\right\} -\bb ^{x}_{t}=|x| \quad \text{a.s.},
\end{align*}
for $\bb ^{x}_{0}=0$ and $\bb ^{x}_{t}=x$ a.s. 
Therefore the condition~\eqref{;mxbvc} in \lref{;mxbv} is fulfilled in such a 
way that, a.s., 
\begin{align*}
 |y-\bb ^{x}_{0}|=|y|\le |x|\le \cp (\bb ^{x})(t)=\cp (\bb ^{x})(t)-\bb ^{x}_{0},
\end{align*}
and hence \eqref{;bv1} holds true. On the event that $\cp (\bb ^{y})(t)\ge |x|$, 
\eqref{;bv2} is a direct consequence of \lref{;mxbv} for $\bb ^{y}_{0}=0$ a.s. 
We also observe that \tref{;tmain2} is consistent with \eqref{;plb}, meaning that 
the theorem entails  
\begin{align*}
 \pr \!\left( 
 \cp (\bb ^{x})(t)\ge |x|
 \right) =\pr \!\left( 
 \cp (\bb ^{y})(t)\ge |x|\rmid| 
 \cp (\bb ^{y})(t)\ge |x|
 \right) .
\end{align*}
As will be seen just above the proof of \cref{;cor3} given at the end of 
the next section, \tref{;tmain2} also implies that \eqref{;concise} 
holds true with $\rho $ replaced by any one-dimensional Brownian bridge 
$\bb$ over $[0,t]$: 
\begin{align}\label{;conciseb}
 \bigl( 
 \cn \bigl( \bbd \bigr) ,\,\cq \bigl( \bbd ) ,\,\bbd 
 \bigr) \eqd 
 \left( 
 \bb ,\,\cq (\bb ),\,\cn (\bb )
 \right) .
\end{align}

We give an outline of the paper. \tref{;tmain2} is proven in \sref{;prftmain2}, in which 
\pref{;disint}, a disintegration formula for Brownian motion by means of Pitman's 
theorem, plays a key role. The proofs of \tref{;tmain1} and its \csref{;cor1} 
and \ref{;cor2} are given in \sref{;prftmain1}; the section is concluded with 
some properties of the transformations $\cn $ and $\cq $. In the Appendix, 
we prove \lref{;mxpl} to be used in the proof of \tref{;tmain2}.

In the sequel, given a positive integer $d$, we write $\ctd{t}{d}$ 
for the space of $\R ^{d}$-valued continuous functions over $[0,t]$. We 
equip $C([0,t];\R ^{d})$ with the topology of uniform convergence, 
and we say that a real-valued function on this space is 
measurable if it is Borel-measurable with respect to the above topology. 
When we say a Brownian bridge, we refer to a one-dimensional one 
unless otherwise specified. 
Given two events $E_{1},E_{2}$, we write $E_{1}=E_{2}$ a.s.\ if their indicator 
functions agree almost surely. Other notation and terminology will be 
introduced as needed.

\section{Proof of \tref{;tmain2}}\label{;prftmain2}

This section is devoted to the proof of \tref{;tmain2}. One of the main ingredients 
of the proof is a disintegration formula for Brownian motion in terms of the transformations 
$\cm _{x}$ and Pitman's transformation $\cp $ defined respectively by \eqref{;cmx} 
and \eqref{;cp}.  

\subsection{Disintegration formula for Brownian motion}\label{;prftmain2s1}

We begin with the following representation of $\cm _{x}$ in terms of $\cp $.

\begin{prop}\label{;mxp}
For every $x\in \R $, we have 
\begin{align}\label{;mxpq1}
 \cm _{x}(\phi )(s)
 =-\cp (\phi )(s)+\min \Bigl\{ 
 2\min _{s\le u\le t}\cp (\phi )(u),\cp (\phi )(t)+x
 \Bigr\} 
\end{align}
for all $0\le s\le t$ and $\phi \in C([0,t];\R )$.
\end{prop}

To prove the above proposition, we prepare a lemma. 
For every $a\in \R $, we write $a_{+}$ for $\max \{ a,0\} $.

\begin{lem}\label{;minp}
For every $\phi \in C([0,t];\R )$, one has 
\begin{align}\label{;minpq1}
 \min _{s\le u\le t}\cp (\phi )(u)
 =\max _{0\le u\le s}\phi _{u}+\Bigl( 
 \max _{0\le u\le s}\phi _{u}-\max _{s\le u\le t}\phi _{u}
 \Bigr) _{+}
\end{align}
for all $0\le s\le t$.
\end{lem}

\begin{proof}
Given $\phi \in C([0,t];\R )$, let 
\begin{align*}
 \sg :=\inf \Bigl\{ 
 s\in [0,t];\,\phi _{s}=\max _{0\le u\le t}\phi _{u}
 \Bigr\} .
\end{align*}
If $s\ge \sg $, then, by the definition of $\cp $, the left-hand side of the claimed 
relation~\eqref{;minpq1} is equal, thanks to the continuity of $\phi $, to 
\begin{align*}
 \min _{s\le u\le t}(2\phi _{\sg }-\phi _{u})
 &=2\phi _{\sg }-\max _{s\le u\le t}\phi _{u}\\
 &=\phi _{\sg }+\Bigl( 
 \phi _{\sg }-\max _{s\le u\le t}\phi _{u}
 \Bigr) _{+},
\end{align*}
which agrees with the right-hand side of \eqref{;minpq1}. 
Next observe that 
\begin{align*}
 \min _{s\le u\le t}\cp (\phi )(u)\ge \max _{0\le u\le s}\phi _{u}
 \quad \text{for any $0\le s\le t$,}
\end{align*}
which is because, by the definition of $\cp $, 
\begin{align*}
 \min _{s\le u\le t}\cp (\phi )(u)&\ge \min _{s\le u\le t}
 \Bigl( 
 2\max _{0\le v\le u}\phi _{v}-\max _{0\le v\le u}\phi _{v}
 \Bigr) \\
 &=\max _{0\le v\le s}\phi _{v}.
\end{align*}
Therefore, in order to complete the proof of the lemma, it suffices to show that 
\begin{align}\label{;minpq3}
 \min _{s\le u\le t}\cp (\phi )(u)\le \max _{0\le u\le s}\phi _{u}
\end{align}
for any $s\le \sg $. If $s\le \sg $ is such that $\phi _{s}=\max _{0\le u\le s}\phi _{u}$, 
then \eqref{;minpq3} holds since, in this case, we have 
\begin{align*}
 \min _{s\le u\le t}\cp (\phi )(u)\le \cp (\phi )(s)=\max _{0\le u\le s}\phi _{u}.
\end{align*}
On the other hand, if $s\le \sg $ is such that $\phi _{s}<\max _{0\le u\le s}\phi _{u}$, 
then, setting 
\begin{align*}
 s^{*}=\inf \Bigl\{ 
 u\in [s,\sg ];\,\phi _{u}=\max _{0\le v\le u}\phi _{v}
 \Bigr\} ,
\end{align*}
we have 
\begin{align*}
 \min _{s\le u\le t}\cp (\phi )(u)&\le \min _{s^{*}\le u\le t}\cp (\phi )(u)\\
 &\le \max _{0\le u\le s^{*}}\phi _{u}\\
 &=\max _{0\le u\le s^{}}\phi _{u},
\end{align*}
where the last two lines follow from the fact that 
\begin{align*}
 \phi _{s^{*}}=\max _{0\le u\le s^{*}}\phi _{u} && \text{and} && 
 \phi _{u}\le \max _{0\le v\le s}\phi _{v}\quad \text{for $s\le u\le s^{*}$}
\end{align*}
by the definition of $s^{*}$. Consequently, \eqref{;minpq3} holds true for 
any $s\le \sg $ and hence the proof of the lemma is completed.
\end{proof}

By the above lemma, \pref{;mxp} follows readily.

\begin{proof}[Proof of \pref{;mxp}]\label{;prfmxp}
Given $0\le s\le t$ and $\phi \in C([0,t];\R )$, by the definition~\eqref{;cp} of 
$\cp $ and by \lref{;minp}, the right-hand side of the claimed relation~\eqref{;mxpq1} 
is equal to 
\begin{align}\label{;mxpq2}
 \phi _{s}+2\min \left\{ 
 \Bigl( 
 \max _{0\le u\le s}\phi _{u}-\max _{s\le u\le t}\phi _{u}
 \Bigr) _{+},\,
 \max _{0\le u\le t}\phi _{u}-\max _{0\le u\le s}\phi _{u}-\frac{\phi _{t}-x}{2}
 \right\} .
\end{align}
Notice that, in the second term,  
\begin{align*}
 \max _{0\le u\le t}\phi _{u}-\max _{0\le u\le s}\phi _{u}
 =\Bigl( 
 \max _{s\le u\le t}\phi _{u}-\max _{0\le u\le s}\phi _{u}
 \Bigr) _{+}; 
\end{align*}
indeed, if we define $\sg \in [0,t]$ as in the proof of the 
previous lemma, then both sides coincide with 
\begin{align*}
 \max _{s\le u\le t}\phi _{u}-\max _{0\le u\le s}\phi _{u}
\end{align*}
when $s\le \sg $, and with $0$ when $s\ge \sg $. The relation 
\begin{align*}
 2\min \{ a,b\} =a+b-|a-b|,\quad a,b\in \R ,
\end{align*}
applied to 
\begin{align*}
 a=\Bigl( 
 \max _{0\le u\le s}\phi _{u}-\max _{s\le u\le t}\phi _{u}
 \Bigr) _{+} && \text{and} && 
 b=\Bigl( 
 \max _{s\le u\le t}\phi _{u}-\max _{0\le u\le s}\phi _{u}
 \Bigr) _{+}-\frac{\phi _{t}-x}{2},
\end{align*}
shows that \eqref{;mxpq2} agrees with $\cm _{x}(\phi )(s)$ 
by the definition~\eqref{;cmx} of $\cm _{x}$.
\end{proof}

Let $B=\{ B_{t}\} _{t\ge 0}$ be a one-dimensional standard Brownian 
motion. Together with the celebrated Pitman's $2M-X$ theorem, 
\pref{;mxp} yields a disintegration formula for Brownian motion in 
terms of $\cm _{x}$ and $\cp $, which is of interest in its own right 
and announced in \cite[Remark~4.4\thetag{1}]{har24+} without proof. 

\begin{prop}\label{;disint}
For every nonnegative measurable function $F$ on $\ctd{t}{}$, 
we have 
\begin{align}\label{;disintq1}
 \ex [F(B)]=\int _{\R }dx\,\ex \!\left[ 
 \frac{1}{2\cp (B)(t)}F\bigl( 
 \cm _{x}(B)
 \bigr) ;\,\cp (B)(t)\ge |x|
 \right] .
\end{align}
\end{prop}

The formula may be compared with \cite[Proposition~2.1]{har25}; see also 
\rref{;rrks}\thetag{1} below. 

Observe that
\begin{align}\label{;iden}
 \cm _{\phi _{t}}(\phi )=\phi \quad \text{for all $\phi \in \ctd{t}{}$},
\end{align}
which, together with \pref{;mxp}, entails that 
\begin{align}\label{;idenp}
 \phi _{s}&=-\cp (\phi )(s)+\min \Bigl\{ 
 2\min _{s\le u\le t}\cp (\phi )(u),\cp (\phi )(t)+\phi _{t}
 \Bigr\} 
\end{align}
for all $0\le s\le t$ and $\phi \in \ctd{t}{}$. Relation~\eqref{;idenp} 
is obtained in \cite[Proposition~1.4]{har24+} by means of Laplace's 
method.

\begin{proof}[Proof of \pref{;disint}]
Let $U$ be a random variable uniformly distributed on $[0,1]$. 
Since the random variable $2U-1$ is uniformly distributed on $[-1,1]$, 
in view of Fubini's theorem, it suffices to show that 
\begin{align}\label{;disintq2}
 \{ B_{s}\} _{0\le s\le t}
 \eqd 
 \left\{ 
 \cm _{(2U-1)\cp (B)(t)}(B)(s)
 \right\} _{0\le s\le t},
\end{align}
provided that $U$ is independent of $B$. To this end, 
for each $a\ge 0$, let $\bes ^{a}=\{ \bes ^{a}_{s}\} _{s\ge 0}$ be a 
three-dimensional Bessel process starting from $a$, and recall 
Pitman's $2M-X$ theorem \cite[Theorem~1.3]{jwp} which states that 
\begin{align}\label{;pit}
 \Bigl\{ 
 \Bigl( \cp (B)(s),\,\max _{0\le u\le s}B_{u}\Bigr) 
 \Bigr\} _{s\ge 0}\eqd 
 \Bigl\{ 
 \Bigl( 
 \bes ^{0}_{s},\,\inf _{u\ge s}\bes ^{0}_{u}
 \Bigr) 
 \Bigr\} _{s\ge 0}
\end{align} 
(see \cite[Chapter~VI, Theorem~\thetag{3.5}]{ry} for the above formulation of 
the theorem). From \eqref{;pit}, we see in particular that 
\begin{align*}
 \left( 
 \{ \cp (B)(s)\} _{0\le s\le t},\,B_{t}
 \right) \eqd 
 \Bigl( 
 \bigl\{ \bes ^{0}_{s}\bigr\} _{0\le s\le t},\,2\inf _{u\ge t}\bes ^{0}_{u}-\bes ^{0}_{t}
 \Bigr) .
\end{align*}
Thanks to the fact that, given $a>0$, $\inf _{s\ge 0}\bes ^{a}_{s}$ 
is distributed as $aU$ (see, e.g., \cite[Chapter~VI, Corollary~\thetag{3.4}]{ry}), 
the Markov property of $\bes ^{0}$ entails that 
the right-hand side of the last identity in law has the same law as 
\begin{align*}
 \bigl( 
 \bigl\{ 
 \bes ^{0}_{s}
 \bigr\} _{0\le s\le t},\,(2U-1)\bes ^{0}_{t}
 \bigr) ,
\end{align*}
where $U$ is independent of $\bes ^{0}$. Therefore, assuming that 
$U$ is independent of $B$, we have 
\begin{align}\label{;disintq3}
 \left( 
 \{ \cp (B)(s)\} _{0\le s\le t},\,B_{t}
 \right) \eqd 
 \bigl( 
 \{ \cp (B)(s)\} _{0\le s\le t},\,(2U-1)\cp (B)(t)
 \bigr) 
\end{align}
since $\bes ^{0}\eqd \cp (B)$ as \eqref{;pit} entails. 
Combining \eqref{;disintq3} with relation~\eqref{;idenp}, we see that 
$\{ B_{s}\} _{0\le s\le t}$ is identical in law with the process 
\begin{align*}
 -\cp (B)(s)+\min \Bigl\{ 
 2\min _{s\le u\le t}\cp (B)(u),\cp (B)(t)+(2U-1)\cp (B)(t)
 \Bigr\} ,\quad 0\le s\le t,
\end{align*}
which proves \eqref{;disintq2} thanks to \pref{;mxp}.
\end{proof}

\begin{rem}
The strong Markov property reveals that identity~\eqref{;disintq2} holds true 
with $t$ replaced by any stopping time $\tau $ of the process 
$\{ \cp (B)(s)\} _{s\ge 0}$ satisfying $\pr (0<\tau <\infty )=1$.
\end{rem}

Notice that formula~\eqref{;disintq1} indicates that 
$\cm _{x}(B)(0)=0$ a.s.\ on the event that $\cp (B)(t)\ge |x|$. 
This is indeed the case as is seen from the following lemma, 
which will be used in the next subsection as well as in 
\ssref{;prfmxpl} where a proof of \lref{;mxpl} is given.

\begin{lem}\label{;mxbv}
Let a pair $(x,\phi )\in \R \times \ctd{t}{}$ be such that 
\begin{align}\label{;mxbvc}
 |x-\phi _{0}|\le \cp (\phi )(t)-\phi _{0}.
\end{align}
Then we have: 
\begin{align*}
 \thetag{i}\ \cm _{x}(\phi )(0)=\phi _{0}; && 
 \thetag{ii}\ \cm _{x}(\phi )(t)=x.
\end{align*}
\end{lem}

Note that the right-hand side of condition~\eqref{;mxbvc} is nonnegative 
since, by the definition of $\cp $, 
\begin{align*}
 \cp (\phi )(t)\ge \max _{0\le u\le t}\phi _{u}\ge \phi _{0}.
\end{align*}

\begin{proof}[Proof of \lref{;mxbv}]
\thetag{i} We have $\min _{0\le u\le t}\cp (\phi )(u)=\phi _{0}$ by \lref{;minp}. 
Therefore, by \pref{;mxp}, 
\begin{align*}
 \cm _{x}(\phi )(0)=-\phi _{0}+\min \left\{ 
 2\phi _{0},\cp (\phi )(t)+x
 \right\} .
\end{align*}
Since condition~\eqref{;mxbvc} entails $2\phi _{0}\le \cp (\phi )(t)+x$, we obtain 
assertion~\thetag{i}.

\thetag{ii} By \pref{;mxp}, 
\begin{align*}
 \cm _{x}(\phi )(t)&=-\cp (\phi )(t)+\min \left\{ 
 2\cp (\phi )(t),\cp (\phi )(t)+x
 \right\} \\
 &=-\cp (\phi )(t)+\cp (\phi )(t)+x\\
 &=x,
\end{align*}
where, for the second line, we have used $x\le \cp (\phi )(t)$ by 
condition~\eqref{;mxbvc}.
\end{proof}

\subsection{Proof of \tref{;tmain2}}\label{;prftmain2s2}

In this subsection, we prove \tref{;tmain2}. Fix a Borel-measurable subset 
$\Ga $ of $\R $ arbitrarily, and consider in \pref{;disint} the function 
$F$ on $\ctd{t}{}$ of the form 
\begin{align*}
 F(\phi )=\ind _{\Ga }(\phi _{t})G(\phi ),\quad \phi \in \ctd{t}{},
\end{align*}
with $G$ a bounded and continuous function. We then have, 
by noting the fact that, thanks to \lref{;mxbv}\thetag{ii}, 
$\cm _{x}(B)(t)=x$ a.s.\ on the event that $\cp (B)(t)\ge |x|$, 
\begin{align*}
 &\int _{\Ga }\frac{dx}{\sqrt{2\pi t}}\,\exp \left( 
 -\frac{x^{2}}{2t} 
 \right) \ex [G(\bb ^{x})]\\
 &=\int _{\Ga }dx\,\ex \!\left[ 
 \frac{1}{2\cp (B)(t)}G(\cm _{x}(B)) ;\,\cp (B)(t)\ge |x|
 \right] .
\end{align*}
Therefore, by the arbitrariness of $\Ga $, we have the relation 
\begin{align}\label{;rel}
 \frac{1}{\sqrt{2\pi t}}\,\exp \left( 
 -\frac{x^{2}}{2t} 
 \right) \ex [G(\bb ^{x})]
 =\ex \!\left[ 
 \frac{1}{2\cp (B)(t)}G(\cm _{x}(B)) ;\,\cp (B)(t)\ge |x|
 \right] ,
\end{align}
which holds true for all $x\in \R $, not for a.e.\ $x\in \R $, thanks to the 
boundedness and continuity of $G$. Indeed, both sides of \eqref{;rel} 
determine a continuous function in $x$, which is immediate from 
the identity in law 
\begin{align*}
 \bb ^{x}\eqd \left\{ 
 \bb ^{0}_{s}+\frac{x}{t}s
 \right\} _{0\le s\le t}
\end{align*}
(see \eqref{;bb}) as to the left-hand side, and from the 
definition~\eqref{;cmx} of $\cm _{x}$ as to the right-hand side. 
With relation~\eqref{;rel} at our disposal for every $x$ and $G$, 
the proof of \tref{;tmain2} follows once we show 

\begin{lem}\label{;mxpl}
For every pair $(x,\phi )\in \R \times \ctd{t}{}$ satisfying the 
condition~\eqref{;mxbvc} in \lref{;mxbv}, we have 
\begin{align}\label{;mxplq1}
 \cp (\cm _{x}(\phi ))=\cp (\phi );
\end{align}
in particular, 
\begin{align}\label{;mxplq2}
 \cm _{y}(\cm _{x}(\phi ))=\cm _{y}(\phi )
\end{align}
for all $y\in \R $.
\end{lem}

The above lemma contains the assertions~\thetag{i} and \thetag{ii} of 
Proposition~3.3 in \cite{har24+}: those two are a special case in which 
$x=-\phi _{t}$ and $y=\phi _{t}$ with $\phi \in \ctd{t}{}$ such that 
$\phi _{0}=0$. We defer the proof of \lref{;mxpl} to the Appendix and 
proceed to the proof of \tref{;tmain2}. 

\begin{proof}[Proof of \tref{;tmain2}]
For every $x\in \R $, by standard density arguments, relation~\eqref{;rel} 
is extended to any nonnegative measurable function $G$. Now we let 
$x,y\in \R $ be such that $|x|\ge |y|$. We apply relation~\eqref{;rel} to 
$\bb ^{y}$ with a function $G$ of the form 
\begin{align*}
 G(\cm _{x}(\phi ))\ind _{[|x|,\infty )}(\cp (\phi )(t)),\quad \phi \in \ctd{t}{},
\end{align*}
where $G$ is again any nonnegative measurable function on $\ctd{t}{}$. 
Then we have 
\begin{equation}\label{;tmain2q3}
\begin{split}
 &\frac{1}{\sqrt{2\pi t}}\,\exp \left( 
 -\frac{y^{2}}{2t} 
 \right) \ex \!\left[ 
 G(\cm _{x}(\bb ^{y}));\,\cp (\bb ^{y})(t)\ge |x|
 \right] \\
 &=\ex \!\left[ 
 \frac{1}{2\cp (B)(t)}G\bigl( \cm _{x}(\cm _{y}(B))\bigr) ;\,\cp (B)(t)\ge \max \{ |x|,|y|\} 
 \right] \\
 &=\ex \!\left[ 
 \frac{1}{2\cp (B)(t)}G(\cm _{x}(B));\,\cp (B)(t)\ge |x| 
 \right] \\
 &=\frac{1}{\sqrt{2\pi t}}\,\exp \left( 
 -\frac{x^{2}}{2t} 
 \right) \ex [G(\bb ^{x})],
\end{split}
\end{equation}
where the third line follows from \eqref{;mxplq2} and $|x|\ge |y|$. 
Dividing the leftmost and rightmost sides of \eqref{;tmain2q3} 
by the identity 
\begin{align}\label{;rks}
 \frac{1}{\sqrt{2\pi t}}\,\exp \left( 
 -\frac{y^{2}}{2t} 
 \right) \pr \!\left( \cp (\bb ^{y})(t)\ge |x|\right) =\frac{1}{\sqrt{2\pi t}}\,\exp \left( 
 -\frac{x^{2}}{2t} 
 \right) ,
\end{align}
we obtain 
\begin{align*}
 \ex \!\left[ 
 G(\cm _{x}(\bb ^{y}))\rmid| \cp (\bb ^{y})(t)\ge |x|
 \right] =\ex [G(\bb ^{x})].
\end{align*}
We let $F$ be a nonnegative measurable function on $\ctd{t}{3}$ and take 
\begin{align*}
 G(\phi )=F(\phi ,\cp (\phi ),\cm _{y}(\phi )),\quad \phi \in \ctd{t}{},
\end{align*}
in the last relation. Then the left-hand side turns into 
\begin{align*}
 &\ex \!\left[ 
 F\bigl( 
 \cm _{x}(\bb ^{y}),\cp (\cm _{x}(\bb ^{y})),\cm _{y}(\cm _{x}(\bb ^{y}))
 \bigr) \rmid| \cp (\bb ^{y})(t)\ge |x|
 \right] \\
 &=\ex \!\left[ 
 F(\cm _{x}(\bb ^{y}),\cp (\bb ^{y}),\cm _{y}(\bb ^{y}))
 \rmid| \cp (\bb ^{y})(t)\ge |x|
 \right] 
\end{align*}
by \eqref{;mxplq1} and \eqref{;mxplq2}. The proof of the theorem 
is completed by noting that $\cm _{y}(\bb ^{y})=\bb ^{y}$ a.s.\ by 
\eqref{;iden}.
\end{proof}

\begin{rem}\label{;rrks}
\thetag{1} We refer to \cite[equation~\thetag{3.2}]{har24} for an analogue 
to relation~\eqref{;tmain2q3} in the framework of 
exponential functionals of Brownian motion.
 
\thetag{2} Since relation~\eqref{;rks} is rewritten as 
\begin{align*}
 \pr \!\left( 
 \max _{0\le u\le t}\bb ^{y}_{u}\ge \frac{|x|+y}{2}
 \right) 
 =\exp \left( 
 -\frac{2}{t}\cdot \frac{|x|+y}{2}\cdot \frac{|x|-y}{2}
 \right) 
\end{align*}
by the definition~\eqref{;cp} of $\cp $, the relation may be seen as 
a consequence of \cite[equation~\thetag{4.3.40}]{ks}; see also 
\cite[equation~\thetag{3.14}]{zam}.
\end{rem}

To apply \tref{;tmain2} to three-dimensional Bessel bridges, it is 
convenient to restate it in terms of the transformations $\cmb _{x}$ and 
$\cpb $, with a Brownian bridge of an arbitrary starting point. 
For this purpose, given $a,b\in \R $, let 
\begin{align*}
 \pbb{a}{b}:=\pr \circ \left( 
 \bb ^{b-a}+a
 \right) ^{-1},
\end{align*}
namely the probability measure $\pbb{a}{b}$ is the law of a Brownian 
bridge from $a$ to $b$ over $[0,t]$, under which we will write $\bb $ for 
the coordinate process in $\ctd{t}{}$: 
\begin{align*}
 \bb _{s}(\phi ):=\phi _{s},\quad 0\le s\le t,\ \phi \in \ctd{t}{}.
\end{align*}
Moreover, $\ebb{a}{b}$ stands for the expectation with respect to 
$\pbb{a}{b}$ below. \tref{;tmain2} is then restated as 

\let\temp\thethm 
\renewcommand{\thethm}{\ref{;tmain2}$'$}
\begin{thm}\label{;tmain2d}
Let $a,b,c\in \R $ be such that 
\begin{align}\label{;cabc}
 \min \{ a,b\} \le c\le \max \{ a,b\} . 
\end{align}
Then it holds that 
\begin{equation}\label{;genbq0}
\begin{split}
 &\ebb{a}{b}\!\left[ 
 F\bigl( \bb -a,\cpb (\bb )+a,\cmb _{a+b-2c}(\bb -a)\bigr) 
 \right] \\
 &=\ebb{c}{a+b-c}\!\left[ 
 F\bigl( 
 \cmb _{b-a}(\bb -c),\cpb (\bb )+c,\bb -c
 \bigr) \rmid| \min _{0\le u\le t}\bb _{u}\le \min \{ a,b\} 
 \right] 
\end{split}
\end{equation}
for any nonnegative measurable function $F$ on $\ctd{t}{3}$.
\end{thm}

\let\thethm\temp 
\addtocounter{thm}{-1}

\begin{proof}
In what follows, $F$ is a generic nonnegative measurable function on 
$\ctd{t}{3}$ and may differ in different contexts. Let $x,y\in \R $ be 
such that $|x|\ge |y|$. First observe that, by \tref{;tmain2}, 
\begin{equation}\label{;genbq1}
\begin{split}
 &\ex \!\left[ 
 F\bigl( \bb ^{x},\cpb (\bb ^{x}),\cmb _{y}(\bb ^{x})\bigr) 
 \right] \\
 &=\ex \!\left[ 
 F\bigl( 
 \cmb _{x}(\bb ^{y}),\cpb (\bb ^{y}),\bb ^{y}
 \bigr) \rmid| \cpb (\bb ^{y})(t)\ge |x|
 \right] .
\end{split}
\end{equation}
Indeed, replacing $x$ and $y$ in \tref{;tmain2} by $-x$ and $-y$, respectively, 
we have 
\begin{align*}
 &\ex \!\left[ 
 F\bigl( \bb ^{-x},\cp (\bb ^{-x}),\cm _{-y}(\bb ^{-x})\bigr) 
 \right] \\
 &=\ex \!\left[ 
 F\bigl( 
 \cm _{-x}(\bb ^{-y}),\cp (\bb ^{-y}),\bb ^{-y}
 \bigr) \rmid| \cp (\bb ^{-y})(t)\ge |x|
 \right] .
\end{align*}
Noting that $\bb ^{-x}\eqd -\bb ^{x}$ as well as $\bb ^{-y}\eqd -\bb ^{y}$, and replacing 
$F$ by a function of the form 
\begin{align*}
 F(-\phi ^{1},\phi ^{2},-\phi ^{3}),\quad 
 (\phi ^{i})_{i=1}^{3}\in \ctd{t}{3},
\end{align*}
we obtain \eqref{;genbq1} in view of the definitions~\eqref{;cmbx} and \eqref{;cpb} 
of the transformations $\cmb _{x}$ and $\cpb $. For $a,b,c\in \R $ satisfying 
condition~\eqref{;cabc}, we are now going to take 
\begin{align*}
 x=b-a && \text{and} && y=a+b-2c,
\end{align*}
which is allowed since, under \eqref{;cabc}, we have $|x|\ge |y|$; indeed, 
\eqref{;cabc} is equivalent to 
\begin{align*}
 a+b-|b-a|\le 2c\le a+b+|b-a|.
\end{align*}
With the above choice of $x$ and $y$ in \eqref{;genbq1}, noting the relations 
\begin{align*}
 \pr \circ (\bb ^{x})^{-1}=\pbb{a}{b}\circ (\bb -a)^{-1}, && 
 \pr \circ (\bb ^{y})^{-1}=\pbb{c}{a+b-c}\circ (\bb -c)^{-1},
\end{align*}
we obtain 
\begin{equation*}
\begin{split}
 &\ebb{a}{b}\!\left[ 
 F\bigl( \bb -a,\cpb (\bb )+a,\cmb _{a+b-2c}(\bb -a)\bigr) 
 \right] \\
 &=\ebb{c}{a+b-c}\!\left[ 
 F\bigl( 
 \cmb _{b-a}(\bb -c),\cpb (\bb )+c,\bb -c
 \bigr) \rmid| \cpb (\bb )(t)+c\ge |b-a|
 \right] .
\end{split}
\end{equation*}
As to the right-hand side, observe that 
\begin{align*}
 \left\{ 
 \cpb (\bb )(t)+c\ge |b-a|
 \right\} 
 =\Bigl\{ 
 \min _{0\le u\le t}\bb _{u}\le \min \{ a,b\} 
 \Bigr\} ,\quad \text{$\pbb{c}{a+b-c}$-a.s.},
\end{align*}
since we have, $\pbb{c}{a+b-c}$-a.s.,     
\begin{align*}
 \cpb (\bb )(t)+c-|b-a|&=a+b-c-2\min _{0\le u\le t}\bb _{u}+c-|b-a|\\
 &=2\Bigl( 
 \min \{ a,b\} -\min _{0\le u\le t}\bb _{u}
 \Bigr) 
\end{align*}
by the definition of $\cpb $. Therefore \eqref{;genbq0} is proven.
\end{proof}

As an immediate corollary to \tref{;tmain2d}, we have \eqref{;conciseb} as below.

\begin{cor}\label{;cor3}
For every $a,b\in \R $, we have 
\begin{align*}
 \ebb{a}{b}\!\left[ 
 F\bigl( \bb ,\cq (\bb ),\cn (\bb )\bigr) 
 \right] 
 =\ebb{b}{a}\!\left[ 
 F\bigl( \cn (\bb ),\cq (\bb ),\bb \bigr) 
 \right] 
\end{align*}
for any nonnegative measurable function $F$ on $\ctd{t}{3}$.
\end{cor}

The fact that 
\begin{align*}
 \pbb{a}{b}\circ \bigl( \bbd \bigr) ^{-1}=\pbb{b}{a}
\end{align*}
entails \eqref{;conciseb}.

\begin{proof}[Proof of \cref{;cor3}]
Given $a,b\in \R $, we are allowed to take $c=b$ in \eqref{;genbq0}. 
Then we have, for any nonnegative measurable function $F$ on $\ctd{t}{3}$, 
\begin{equation*}
\begin{split}
 &\ebb{a}{b}\!\left[ 
 F\bigl( \bb -a,\cpb (\bb )+a,\cmb _{a-b}(\bb -a)\bigr) 
 \right] \\
 &=\ebb{b}{a}\!\left[ 
 F\bigl( 
 \cmb _{b-a}(\bb -b),\cpb (\bb )+b,\bb -b
 \bigr) \right] ,
\end{split}
\end{equation*}
since 
\begin{align*}
 \min _{0\le u\le t}\bb _{u}\le \min \{ a,b\} ,\quad \text{$\pbb{b}{a}$-a.s.}
\end{align*}
Replacing $F$ by a nonnegative measurable function of the form 
\begin{align*}
 F(\phi ^{1}+a,\phi ^{2},\phi ^{3}+b),\quad 
 (\phi ^{i})_{i=1}^{3}\in \ctd{t}{3},
\end{align*}
we have the claim by the definitions~\eqref{;cn} and \eqref{;cq} of 
the transformations $\cn $ and $\cq $. 
\end{proof}

\section{Proofs of \tref{;tmain1} and \csref{;cor1} and \ref{;cor2}}\label{;prftmain1}

In this section, we prove \tref{;tmain1} and its \csref{;cor1} and 
\ref{;cor2}. Counterparts to those corollaries in the case of Brownian motion are  
also discussed and some properties of the transformations $\cn $ and $\cq $ 
are summarized as well.

\subsection{Proof of \tref{;tmain1}}\label{;prftmain1s}
In the sequel, given $a,b\ge 0$, $\pbes{a}{b}$ denotes the 
law of a three-dimensional Bessel bridge from $a$ to $b$ over $[0,t]$, 
under which we write $\rho $ for the coordinate process in $\ctd{t}{}$: 
\begin{align*}
 \rho _{s}(\phi ):=\phi _{s},\quad 0\le s\le t,\ \phi \in \ctd{t}{}.
\end{align*}
We also write $\ebes{a}{b}$ for the expectation with respect to $\pbes{a}{b}$.
We will prove the theorem in a greater generality. 

\begin{thm}\label{;genr}
Given $a,b,c>0$ satisfying the condition~\eqref{;cabc} in \tref{;tmain2d}, 
we have 
\begin{equation}\label{;genrq0}
\begin{split}
 &\ebes{a}{b}\!\left[ 
 F\bigl( \rho -a,\cpb (\rho )+a,\cmb _{a+b-2c}(\rho -a)\bigr) 
 \right] \\
 &=\ebes{c}{a+b-c}\!\left[ 
 F\bigl( 
 \cmb _{b-a}(\rho -c),\cpb (\rho )+c,\rho -c
 \bigr) \rmid| \min _{0\le u\le t}\rho _{u}\le \min \{ a,b\} 
 \right] 
\end{split}
\end{equation}
for any nonnegative measurable function $F$ on 
$\ctd{t}{3}$.
\end{thm}

\begin{proof}
Since we have \eqref{;genbq0} for any nonnegative 
measurable function $F$ by the assumption on $a,b,c$, 
we replace $F$ by a function of the form 
\begin{align*}
 F(\phi ^{1},\phi ^{2},\phi ^{3})\ind _{(-\infty ,a+b]}(\phi ^{2}_{t}),\quad 
 (\phi ^{i})_{i=1}^{3}\in \ctd{t}{3},
\end{align*}
where $F$ is again any nonnegative measurable function on $\ctd{t}{3}$. 
Observe that, $\pbb{a}{b}$-a.s., 
\begin{align*}
 \left\{ 
 \cpb (\bb )(t)+a\le a+b
 \right\} &=\Bigl\{ 
 b-2\min _{0\le u\le t}\bb _{u}\le b
 \Bigr\} \\
 &=\Bigl\{ 
 \min _{0\le u\le t}\bb _{u}\ge 0
 \Bigr\} ,
\intertext{and that, $\pbb{c}{a+b-c}$-a.s.,}
 \left\{ 
 \cpb (\bb )(t)+c\le a+b
 \right\} &=\Bigl\{ 
 a+b-c-2\min _{0\le u\le t}\bb _{u}+c\le a+b
 \Bigr\} \\
 &=\Bigl\{ 
 \min _{0\le u\le t}\bb _{u}\ge 0
 \Bigr\} ,
\end{align*}
by the definition~\eqref{;cpb} of $\cpb$. Hence, with the above 
replacement of $F$, \eqref{;genbq0} becomes 
\begin{equation*}
\begin{split}
 &\ebb{a}{b}\!\left[ 
 F\bigl( \bb -a,\cpb (\bb )+a,\cmb _{a+b-2c}(\bb -a)\bigr) 
 ;\,\min _{0\le u\le t}\bb _{u}\ge 0\right] \\
 &=\ebb{c}{a+b-c}\!\left[ 
 F\bigl( 
 \cmb _{b-a}(\bb -c),\cpb (\bb )+c,\bb -c
 \bigr) ;\,\min _{0\le u\le t}\bb _{u}\ge 0\rmid| \min _{0\le u\le t}\bb _{u}\le \min \{ a,b\} 
 \right] .
\end{split}
\end{equation*}
Recall the well-known fact that, for every $a,b>0$, 
a Brownian bridge from $a$ to $b$ over $[0,t]$ conditioned 
to stay nonnegative has the law $\pbes{a}{b}$ (see, e.g., 
\cite[Lemma~3.22]{zam} and \cite[Lemma~5.2.8]{kni}). 
Dividing both sides of the last equality by the identity 
\begin{align}\label{;genrq1}
 \pbb{a}{b}\!\left( 
 \min _{0\le u\le t}\bb _{u}\ge 0
 \right) 
 =\pbb{c}{a+b-c}\!\left( 
 \min _{0\le u\le t}\bb _{u}\ge 0
 \rmid| \min _{0\le u\le t}\bb _{u}\le \min \{a, b\}
 \right) ,
\end{align}
we obtain \eqref{;genrq0} thanks to the above-mentioned fact. Here we used 
the rewriting of the right-hand side of \eqref{;genrq1} in the form 
\begin{align*}
 \pbes{c}{a+b-c}\!\left( 
 \min _{0\le u\le t}\bb _{u}\le \min \{ a,b\} 
 \right) 
 \frac{
 \pbb{c}{a+b-c}\!\left( 
 \min _{0\le u\le t}\bb _{u}\ge 0
 \right) 
 }{
 \pbb{c}{a+b-c}\!\left( 
 \min _{0\le u\le t}\bb _{u}\le \min \{ a,b\} 
 \right) 
 }
\end{align*}
by Bayes' rule. 
\end{proof}

\tref{;genr} entails \tref{;tmain1} readily. 

\begin{proof}[Proof of \tref{;tmain1}]
Suppose first that $a,b>0$. The argument is the same as in 
the proof of \cref{;cor3}: we are allowed to take $c=b$ in \tref{;genr}, 
whence, by noting that 
\begin{align*}
 \min _{0\le u\le t}\rho _{u}\le \min \{ a,b\} ,\quad \text{$\pbes{b}{a}$-a.s.},
\end{align*}
we have 
\begin{align*}
 &\ebes{a}{b}\!\left[ 
 F\bigl( 
 \rho -a,\cpb (\rho )+a,\cmb _{a-b}(\rho -a)
 \bigr) 
 \right] \\
 &=\ebes{b}{a}\!\left[ 
 F\bigl( 
 \cmb _{b-a}(\rho -b),\cpb (\rho )+b,\rho -b
 \bigr) 
 \right] 
\end{align*}
for every nonnegative measurable function $F$ on 
$\ctd{t}{3}$. Due to the fact that 
\begin{align}\label{;rev}
 \pbes{a}{b}\circ (\rd )^{-1}=\pbes{b}{a}\quad 
 \text{for any $a,b\ge 0$}
\end{align}
(see, e.g., \cite[Chapter~XI, Exercise~\thetag{3.7}]{ry}), we obtain 
the assertion of the theorem when $a,b>0$. To deal with the case 
$ab=0$, note that 
\begin{align*}
 \pbes{a}{b}\Rightarrow \pbes{a}{0}\quad (b\downarrow 0,a>0), && 
 \pbes{a}{b}\Rightarrow \pbes{0}{b}\quad (a\downarrow 0,b>0), && 
 \pbes{a}{0}\Rightarrow \pbes{0}{0}\quad (a\downarrow 0),
\end{align*}
where the notation $\Rightarrow $ stands for weak convergence. 
The first two are seen from \cite[Subsection~2.2]{cub} and \eqref{;rev}. 
We know that, for a three-dimensional Brownian bridge 
$\gamma =\left\{ 
\gamma _{s}=\bigl( \gamma ^{1}_{s},\gamma ^{2}_{s},\gamma ^{3}_{s}\bigr) 
\right\} _{0\le s\le t}$ over $[0,t]$ null at both ends, the process given by 
\begin{align*}
 \sqrt{\{ a(1-s/t)+\gamma ^{1}_{s}\} ^{2}+(\gamma ^{2}_{s})^{2}+(\gamma ^{3}_{s})^{2}},\quad 
 0\le s\le t,
\end{align*}
has the law $\pbes{a}{0}$ for every $a\ge 0$ (see, e.g., \cite{yz}), from 
which we conclude the third easily. 
Therefore the case $ab=0$ is proven by noting that the two 
functions $\cn $ and $\cq $ defined respectively by \eqref{;cn} and 
\eqref{;cq} are both continuous (in fact, easily seen to be Lipschitz 
continuous) on $\ctd{t}{}$. This completes the proof of the theorem.
\end{proof}

\begin{rem}\label{;rwc}
At the beginning of \cite[Chapter~XI, Section~3]{ry}, it is noted that 
the mapping 
\begin{align*}
 [0,\infty )^{2}\ni (a,b)\mapsto \pbes{a}{b}
\end{align*}
is continuous in the topology of weak convergence. However, we 
could not find a reference in which the above-mentioned continuity 
is proven and we have chosen to argue indirectly in the last proof to deal with 
the case $ab=0$.
\end{rem}

\subsection{Proofs of \csref{;cor1} and \ref{;cor2}}\label{;prfcors12}

We begin with some facts about the transformations $\cn $ and $\cq $ when they 
are restricted to nonnegative functions.

\begin{lem}\label{;fnq}
Given $\phi \in \ctd{t}{}$ such that $\phi _{s}\ge 0$ for all $0\le s\le t$, 
we have the following:
\begin{enumerate}[(i)]
\item when $\phi _{0}=0$, 
\begin{align*}
 \overbrace{\cn (\phi )}^{\circ }=\cq \bigl( \rphi \bigr) , && 
 \cq (\phi )=\phi ;
\end{align*}
\item when $\phi _{t}=0$, 
\begin{align*}
 \cn (\phi )=\cq (\phi ).
\end{align*}
\end{enumerate}
\end{lem}

\begin{proof}
\thetag{i} By the assumption that $\phi $ is nonnegative, we have 
$\min _{0\le u\le s}\phi _{u}=0$ for all $0\le s\le t$. Therefore, 
from the expression~\eqref{;cnexpr} of $\cn $, we see that, for every $0\le s\le t$, 
\begin{align*}
 \cn (\phi )(s)&=\phi _{s}+\Bigl| 
 \phi _{t}-\min _{s\le u\le t}\phi _{u}\Bigr| -\min _{s\le u\le t}\phi _{u}\\
 &=\phi _{s}+\phi _{t}-2\min _{s\le u\le t}\phi _{u},
\end{align*}
whence 
\begin{equation}\label{;qinv}
\begin{split}
 \overbrace{\cn (\phi )}^{\circ }(s)&=\phi _{t-s}+\phi _{t}-2\min _{t-s\le u\le t}\phi _{u}\\
 &=\cq \bigl( \rphi \bigr) (s)
\end{split}
\end{equation}
by the definition~\eqref{;cq} of 
$\cq $. The latter is obvious from the definition of $\cq $.

\thetag{ii} In this case, we see that, for every $0\le s\le t$, 
\begin{align*}
 \cn (\phi )(s)&=\phi _{s}+\Bigl| 
 -\phi _{0}+\min _{0\le u\le s}\phi _{u}
 \Bigr| -\min _{0\le u\le s}\phi _{u}\\
 &=\phi _{s}+\phi _{0}-2\min _{0\le u\le s}\phi _{u},
\end{align*}
which is $\cq (\phi )(s)$ as claimed.
\end{proof}

For every $a\ge 0$, let $\pBES{a}$ denote the 
law of a three-dimensional Bessel process starting from $a$, 
under which we write $R$ for the coordinate process in $C([0,\infty );\R )$: 
\begin{align*}
 R_{s}(\phi ):=\phi _{s},\quad s\ge 0,\ \phi \in C([0,\infty );\R ).
\end{align*}

\begin{proof}[Proof of \cref{;cor1}]
Let $F$ be a nonnegative measurable function on $\ctd{t}{3}$. Then, by 
\tref{;tmain1}, or its restatement~\eqref{;concise}, we have, for every $b\ge 0$,
\begin{align*}
 \ebes{a}{b}\!\left[ 
 F\bigl( 
 \cn (\rd ),\cq (\rd ),\rd 
 \bigr) 
 \right] 
 =\ebes{a}{b}\!\left[ 
 F\bigl( 
 \rho ,\cq (\rho ),\cn (\rho )
 \bigr) 
 \right] .
\end{align*}
Integrating both sides with respect to the law $\pBES{a}(R_{t}\in db)$ over 
$[0,\infty )$, we have the former claim~\eqref{;cor1q1}. As for the latter, note that, 
$\pBES{0}$-a.s., $R_{s}\ge 0$ for all $s\ge 0$ and $\Rd _{t}=R_{0}=0$. 
Therefore we may apply \thetag{ii} (resp.\ \thetag{i}) of 
\lref{;fnq} to the left-hand (resp.\ right-hand) side of \eqref{;cor1q1}. 
Then, identity~\eqref{;cor1q1} reduces to 
\begin{align*}
 \Bigl( 
 \cq \bigl( \Rd \bigr) ,\,\Rd 
 \Bigr) \eqd 
 \Bigl( 
 R,\,\overbrace{\cq \bigl( \Rd \bigr) }^{\circ }
 \Bigr) .
\end{align*}
Noting the expression of $\cq \bigl( \Rd \bigr) $ as in \eqref{;qinv}, 
we also obtain the latter claim~\eqref{;cor1q2}. 
\end{proof}

We turn to the proof of \cref{;cor2}.
For every $\phi \in C([0,\infty );\R )$ satisfying 
\begin{align}\label{;conphi}
 \lim _{s\to \infty }\frac{\phi _{s}}{1+s}=0, 
\end{align}
set $\iota (\phi ):[0,1]\to \R $ by
\begin{align}\label{;defio}
 \iota (\phi )(s):=
 \begin{cases}
 (1-s)\phi _{s/(1-s)}, & 0\le s<1,\\
 0, & s=1.
 \end{cases}
\end{align}
Then, $\iota (\phi )\in C([0,1];\R )$ by assumption. We take $t=1$ below 
until the end of the proof of \cref{;cor2}.

\begin{lem}\label{;linv}
Suppose that $\phi \in C([0,\infty );\R )$ fulfills condition~\eqref{;conphi}. 
Then it holds that 
\begin{align}
 \iota (\phi )(s/(1+s))&=\frac{1}{1+s}\phi _{s}, \label{;linvq1}\\
 \cq (\iota (\phi ))(s/(1+s))&=\frac{1}{1+s}\phi _{s}-2\min _{0\le u\le s}
 \frac{\phi _{u}}{1+u}+\phi _{0}, \label{;linvq2}
\intertext{for all $s\ge 0$, and that}
 \overbrace{\iota (\phi )}^{\circ }(s/(1+s))
 &=\frac{s}{1+s}\phi _{1/s}, \label{;linvq3}\\
 \overbrace{\cq (\iota (\phi ))}^{\circ }(s/(1+s))
 &=\frac{s}{1+s}\phi _{1/s}-2\min _{0\le u\le 1/s}
 \frac{\phi _{u}}{1+u}+\phi _{0}, \label{;linvq4}
\end{align}
for all $s>0$.
\end{lem}

\begin{proof}
Assertions~\eqref{;linvq1} and \eqref{;linvq2} are immediate from the 
definitions~\eqref{;defio} and \eqref{;cq} of $\iota (\phi )$ and $\cq $. 
By noting the definition~\eqref{;ts} of the operation of time reversal with 
$t=1$, and that 
\begin{align*}
 1-\frac{s}{1+s}=\frac{(1/s)}{1+(1/s)}
\end{align*}
for $s>0$, assertions~\eqref{;linvq3} and \eqref{;linvq4} are obtained 
from \eqref{;linvq1} and \eqref{;linvq2}, respectively, with replacing 
$s$ by $1/s$.
\end{proof}

\begin{proof}[Proof of \cref{;cor2}]
By the laws of the iterated logarithm for three-dimensional Bessel processes 
(see, e.g., \cite[p.~77]{bs}), 
\begin{align*}
 \pBES{a}\!\left( 
 \lim _{s\to \infty }\frac{R_{s}}{1+s}=0
 \right) =1.
\end{align*}
Hence $\iota (R)$ determines a continuous process over $[0,1]$, and we know 
from \eqref{;inver} that 
\begin{align}\label{;cor2q3}
 \pBES{a}\circ (\iota (R))^{-1}=\pbes{a}{0}.
\end{align}
Observe that, under $\pbes{a}{0}$, \tref{;tmain1}, or its restatement~\eqref{;concise}, 
reduces to 
\begin{align}\label{;cor2q4}
 \Bigl( 
 \overbrace{\cq (\rho )}^{\circ },\,
 \rd 
 \Bigr) \eqd 
 \left( \rho ,\,\cq (\rho )\right) 
\end{align}
by \lref{;fnq}. Indeed, for $\rd _{0}=0$, $\pbes{a}{0}$-a.s., we have 
\begin{align*}
 \cn \bigl( \rd \bigr) =\overbrace{\cq (\rho )}^{\circ }, && 
 \cq \bigl( \rd \bigr) =\rd ,
\end{align*}
$\pbes{a}{0}$-a.s., by \lref{;fnq}\thetag{i}, and for $\rho _{1}=0$, $\pbes{a}{0}$-a.s., 
we also have 
\begin{align*}
 \cn (\rho )=\cq (\rho ),\quad \text{$\pbes{a}{0}$-a.s.,}
\end{align*} 
by \lref{;fnq}\thetag{ii}.
Now combining \eqref{;cor2q3} and \eqref{;cor2q4} leads to 
\begin{align*}
 \Bigl( 
 \overbrace{\cq (\iota (R))}^{\circ },\,\overbrace{\iota (R)}^{\circ }
 \Bigr) 
 \eqd 
 \bigl( 
 \iota (R),\,\cq (\iota (R))
 \bigr) 
\end{align*}
and the assertion of \cref{;cor2} follows from \lref{;linv}.
\end{proof}

It is of interest to apply similar reasoning to the above proofs of 
\csref{;cor1} and \ref{;cor2} to Brownian motion. For every $a\in \R $, 
define $\pr _{a}$ to be the law of a one-dimensional Brownian motion 
starting from $a$, under which we write $B$ for the coordinate process 
in $C([0,\infty );\R )$: 
\begin{align*}
 B_{s}(\phi ):=\phi _{s},\quad s\ge 0,\ \phi \in C([0,\infty );\R ).
\end{align*}
By identity~\eqref{;conciseb}, we see that 
\begin{align}
 \left( 
 \cn \bigl( \mathring{B}\bigr) ,\,
 \cq \bigl( \mathring{B}\bigr) ,\,
 \mathring{B}
 \right) &\eqd 
 \left( 
 B,\,\cq (B),\,\cn (B)
 \right) ,\label{;bts}
\intertext{as well as that, with $t=1$,}
 \biggl( 
 \cn \Bigl( 
 \overbrace{\iota (B)}^{\circ }
 \Bigr) ,\,\cq \Bigl( 
 \overbrace{\iota (B)}^{\circ }
 \Bigr) ,\,\overbrace{\iota (B)}^{\circ }
 \biggr) 
 &\eqd 
 \bigl( 
 \iota (B),\,\cq (\iota (B)),\,\cn (\iota (B))
 \bigr) .\label{;binv}
\end{align}
For every $a\in \R $, note that, under $\pr _{a}$, the 
following two identities in law hold by the time-reversal and 
time-inversion of Brownian motion:
\begin{align*}
 \mathring{B}&\eqd \{ B_{s}-B_{t}+a\} _{0\le s\le t},\\
 \{ sB_{1/s}\} _{s>0}&\eqd 
 \{ -B_{s}+a(1+s)\} _{s>0}.
\end{align*}
If we apply the former identity to \eqref{;bts}, then we obtain  
\begin{align}\label{;btsd}
 \left( 
 \cn (B)-B_{t}+a,\,\cpb (B),\,B
 \right) 
 \eqd 
 \left( 
 B,\,\cpb (B),\,\cn (B)-B_{t}+a
 \right) ,
\end{align}
the case $a=0$ of which recovers \cite[Corollary~1.3]{har24+} with 
$\mu =0$ in view of the symmetry of Brownian motion. On the other hand, 
if we apply the latter to \eqref{;binv}, we obtain \tref{;binvthm} below; 
details are left to the reader. 

\begin{thm}\label{;binvthm}
For every $\phi \in C([0,\infty );\R )$ satisfying $\lim _{s\to \infty }\phi _{s}=0$, 
set 
\begin{align*}
 \cs (\phi )(s):=
 \phi _{s}-\phi _{0}+\Bigl| -\phi _{0}+
 \min _{0\le u\le s}\phi _{u}-\inf _{u\ge s}\phi _{u}
 \Bigr| 
 -\Bigl| 
 \min _{0\le u\le s}\phi _{u}-\inf _{u\ge s}\phi _{u}
 \Bigr| ,\quad s\ge 0.
\end{align*}
Define the process $Y=\{ Y_{s}\} _{s\ge 0}$ by  
\begin{align*}
 Y_{s}:=\frac{B_{s}}{1+s},\quad s\ge 0.
\end{align*}
Then, for every $a\in \R $, it holds that 
\begin{align*}
 \left( 
 \cs (-Y),\,\cq (-Y),\,-Y
 \right) 
 \eqd 
 \left( 
 Y,\,\cq (Y),\,\cs (Y)
 \right) 
\end{align*}
under $\pr _{a}$.
\end{thm}

The above theorem entails in particular the following invariance 
in law of the standard Brownian motion: under $\pr _{0}$,
\begin{equation*}
\begin{split}
 \{ B_{s}\} _{s\ge 0}
 \eqd \Biggl\{ 
 &B_{s}+a(1-s)+(1+s)\left| 
 -a+\min _{0\le u\le s}\frac{a+B_{u}}{1+u}-\inf _{u\ge s}\frac{a+B_{u}}{1+u}
 \right| \\
 &-(1+s)\left| 
 \min _{0\le u\le s}\frac{a+B_{u}}{1+u}-\inf _{u\ge s}\frac{a+B_{u}}{1+u}
 \right| 
 \Biggr\} _{s\ge 0}
\end{split}
\end{equation*}
for any $a\in \R $. Noting that 
\begin{align*}
 \inf _{u\ge 0}\frac{a+B_{u}}{1+u}\le 0,\quad \text{$\pr _{0}$-a.s.},
\end{align*}
because, in view of the laws of the iterated logarithm of Brownian motion, 
$\lim _{u\to \infty }(a+B_{u})/(1+u)=0$, $\pr _{0}$-a.s., we see that 
the process on the right-hand side of the above identity in law does 
start from the origin: 
\begin{align*}
 &a+\left| 
 -a+a-\inf _{u\ge 0}\frac{a+B_{u}}{1+u}
 \right| -\left| 
 a-\inf _{u\ge 0}\frac{a+B_{u}}{1+u}
 \right| \\
 &=a-\inf _{u\ge 0}\frac{a+B_{u}}{1+u}-\left( 
 a-\inf _{u\ge 0}\frac{a+B_{u}}{1+u}
 \right) \\
 &=0,
\end{align*}
$\pr _{0}$-a.s.

\subsection{Properties of $\cn $ and $\cq $}\label{;sspnq}

We summarize here some properties of the transformations $\cn$ and $\cq $ 
as announced in \sref{;intro}. 

\begin{prop}\label{;pnq}
The transformations $\cn $ and $\cq $ defined respectively by 
\eqref{;cn} and \eqref{;cq} enjoy the following properties: 
\begin{enumerate}[(i)]{}
\item for every $\phi \in \ctd{t}{}$,
\begin{align*}
 \cn (\phi )(0)=\phi _{t}, && \cn (\phi )(t)=\phi _{0};
\end{align*}

\item it holds that 
\begin{align*}
\cn \circ \cn =\id && \text{and}  && \cq \circ \cn =\cq , 
\end{align*}
where 
$\id $ is the identity map on $\ctd{t}{}$ as in \sref{;intro}; 

\item for every $\phi \in \ctd{t}{}$, 
\begin{align*}
 \overbrace{\cn (\phi )}^{\circ }=\cn \bigl( \rphi \bigr) .
\end{align*}
\end{enumerate}
\end{prop}

Note the following expressions of the two transformations by the 
definitions~\eqref{;cmbx} and \eqref{;cpb} of $\cmb _{x}$ and $\cpb $: 
\begin{align}
 \cn (\phi )&=-\cm _{\phi _{t}-\phi _{0}}(\phi _{0}-\phi )+\phi _{t}, \label{;xcn}\\
 \cq (\phi )&=\cp (\phi _{0}-\phi ), \label{;xcq}
\end{align}
for every $\phi \in \ctd{t}{}$.

\begin{proof}[Proof of \pref{;pnq}]
\thetag{i} By the fact that $\min _{0\le u\le t}\phi _{u}\le \min \{ \phi _{0},\phi _{t}\} $, 
expression \eqref{;cnexpr} entails that 
\begin{align*}
 \cn (\phi )(0)&=\phi _{0}+\phi _{t}-\min _{0\le u\le t}\phi _{u}
 -\Bigl( \phi _{0}-\min _{0\le u\le t}\phi _{u}\Bigr) \\
 &=\phi _{t},\\
 \cn (\phi )(t)&=\phi _{t}+\phi _{0}-\min _{0\le u\le t}\phi _{u}
 -\Bigl( \phi _{t}-\min _{0\le u\le t}\phi _{u}\Bigr) \\
 &=\phi _{0},
\end{align*}
as claimed. 

\thetag{ii} Fix $\phi \in \ctd{t}{}$ below. We appeal to \lref{;mxpl}. Notice that 
the pair 
\begin{align*}
 (\phi _{t}-\phi _{0},\phi _{0}-\phi )\in \R \times \ctd{t}{}
\end{align*}
fulfills condition~\eqref{;mxbvc}; indeed, for the function 
$\phi _{0}-\phi $ is null at the origin, with the above choice of a pair, 
\eqref{;mxbvc} reads
\begin{align*}
 |\phi _{t}-\phi _{0}|\le \phi _{t}+\phi _{0}-2\min _{0\le u\le t}\phi _{u},
\end{align*}
which is the case because 
\begin{align*}
 2\min _{0\le u\le t}\phi _{u}&\le 2\min \{ \phi _{t},\phi _{0}\} \\
 &=\phi _{t}+\phi _{0}-|\phi _{t}-\phi _{0}|.
\end{align*}
Noting \thetag{i}, we see from \eqref{;xcn} that 
\begin{align*}
 (\cn \circ \cn )(\phi )
 &=-\cm _{\phi _{0}-\phi _{t}}\bigl( \phi _{t}-\cn (\phi )\bigr) +\phi _{0}\\
 &=-\cm _{\phi _{0}-\phi _{t}}\bigl( \cm _{\phi _{t}-\phi _{0}}(\phi _{0}-\phi )\bigr) +\phi _{0}\\
 &=-\cm _{\phi _{0}-\phi _{t}}(\phi _{0}-\phi )+\phi _{0}\\
 &=-(\phi _{0}-\phi )+\phi _{0},
\end{align*}
where we used \lref{;mxpl} for the third line and \eqref{;iden} for the fourth. 
Similarly, we have, by \eqref{;xcq}, 
\begin{align*}
 (\cq \circ \cn )(\phi )&=\cp \bigl( 
 \phi _{t}-\cn (\phi )
 \bigr) \\
 &=\cp \bigl( 
 \cm _{\phi _{t}-\phi _{0}}(\phi _{0}-\phi )
 \bigr) \\
 &=\cp (\phi _{0}-\phi ),
\end{align*}
which is $\cq (\phi )$ again by \eqref{;xcq}, where we used \lref{;mxpl} for the 
third line.

\thetag{iii} Fix $0\le s\le t$ arbitrarily. By the definition~\eqref{;ts} of the 
operation of time reversal and by the expression~\eqref{;cnexpr} of $\cn $,
\begin{align*}
 \overbrace{\cn (\phi )}^{\circ }(s)
 &=\cn (\phi )(t-s)\\
 &=\phi _{t-s}+\Bigl| \phi _{t}-\phi _{0}+
 \min _{0\le u\le t-s}\phi _{u}-\min _{t-s\le u\le t}\phi _{u}
 \Bigr| 
 -\Bigl| 
 \min _{0\le u\le t-s}\phi _{u}-\min _{t-s\le u\le t}\phi _{u}
 \Bigr| \\
 &=\rphi _{s}+\Bigl| \rphi _{0}-\rphi _{t}+
 \min _{s\le u\le t}\rphi _{u}-\min _{0\le u\le s}\rphi _{u}
 \Bigr| 
 -\Bigl| 
 \min _{s\le u\le t}\rphi _{u}-\min _{0\le u\le s}\rphi _{u}
 \Bigr| \\
 &=\cn \bigl( \rphi \bigr) (s)
\end{align*}
as expected.
\end{proof}

\appendix 
\section*{Appendix}
\renewcommand{\thesection}{A}
\setcounter{equation}{0}
\setcounter{thm}{0}

The purpose of this Appendix is to prove \lref{;mxpl}; two fundamental lemmas 
in calculus used in the proof are also appended with their proofs for the sake 
of the reader's convenience.

\subsection{Proof of \lref{;mxpl}}\label{;prfmxpl}

We appeal to Laplace's method along the same lines as in the proof of 
\cite[Proposition~3.3]{har24+}. For this purpose, we prepare extra two 
transformations $A$ and $Z$ on $C([0,\infty );\R )$ defined respectively by 
\begin{align*}
 A_{s}(\phi ):=\int _{0}^{s}e^{2\phi _{u}}\,du, &&  
 Z_{s}(\phi ):=e^{-\phi _{s}}A_{s}(\phi ),
\end{align*}
for $s\ge 0$ and $\phi \in C([0,\infty );\R )$. These two transformations 
are related via 
\begin{align}\label{;relaz}
 \frac{d}{ds}\frac{1}{A_{s}(\phi )}=-\frac{1}{\{Z_{s}(\phi )\}^{2}},\quad s>0,
\end{align}
for every $\phi \in C([0,\infty );\R )$. Keeping the same notation, we restrict 
$A$ to $\ctd{t}{}$, with which we define 
$\ctx :\ctd{t}{}\to \ctd{t}{}$ by 
\begin{equation}\label{;ctx}
\begin{split}
 &\ctx (\phi )(s)\\
 &:=\phi _{s}-\frac{1}{c}\log \left\{ 
 1+\frac{A_{s}(c\phi )}{A_{t}(c\phi )}\left( 
 e^{c\phi _{t}-cx}-1
 \right) 
 \right\} ,\quad 0\le s\le t,\ \phi \in \ctd{t}{},
\end{split}
\end{equation}
for each $c>0$ and $x\in \R $. Restricting $Z$ to $\ctd{t}{}$ as well, we have 

\begin{lem}\label{;ctxl}
For every $c>0$, $x\in \R $ and $\phi \in \ctd{t}{}$, it holds that 
\begin{align}\label{;ctxlq1}
 Z_{s}\bigl( c\ctx (\phi )\bigr) =Z_{s}(c\phi ),\quad 0\le s\le t.
\end{align}
Moreover, if the pair $(x,\phi )$ fulfills the condition~\eqref{;mxbvc} in \lref{;mxbv}, 
then, as $c\to \infty $,
\begin{align}\label{;ctxlq2}
 \max _{0\le u\le t}\bigl| 
 \ctx (\phi )(u)-\cm _{x}(\phi )(u)
 \bigr| \to 0.
\end{align}
\end{lem}

\begin{proof}
A direct computation shows that 
\begin{align*}
 \frac{1}{A_{s}\bigl( c\ctx (\phi )\bigr) }=\frac{1}{A_{s}(c\phi )}
 +\frac{e^{c\phi _{t}-cx}-1}{A_{t}(c\phi )}
\end{align*}
for $0<s\le t$. Differentiating both sides with respect to $s$ and noting 
relation~\eqref{;relaz}, we obtain \eqref{;ctxlq1}. As for the latter claim~\eqref{;ctxlq2}, 
notice that, for each $c>0$, the function $\ctx (\phi )-\phi $ is monotone 
(in fact, nonincreasing if $\phi _{t}-x\ge 0$ and increasing otherwise) 
and that the claimed limit function $\cm _{x}(\phi )-\phi $ is continuous. 
Therefore, in order to prove \eqref{;ctxlq2}, it suffices to check the pointwise 
convergence in virtue of Dini's second theorem; see 
\cite[pp.\ 81 and 270, Problem~127]{ps}, as well as \lref{;tdini} below. 
To this end, by the definition~\eqref{;ctx} of $\ctx $, 
\begin{align*}
 \ctx (\phi )(0)=\phi _{0},&& \ctx (\phi )(t)=x,
\end{align*}
which, by \lref{;mxbv}, agree with $\cm _{x}(\phi )(0)$ and $\cm _{x}(\phi )(t)$, 
respectively. Next we pick $s\in (0,t)$. In \eqref{;ctx}, rewrite 
\begin{align*}
1+\frac{A_{s}(c\phi )}{A_{t}(c\phi )}\left( 
 e^{c\phi _{t}-cx}-1
 \right) 
 =\frac{1}{A_{t}(c\phi )}\left( 
 e^{c\phi _{t}-cx}A_{s}(c\phi )+\int _{s}^{t}e^{2c\phi _{u}}\,du
 \right) ,
\end{align*}
and observe that, as $c\to \infty $, 
\begin{align*}
 \frac{1}{c}\log 
 \left\{ e^{c\phi _{t}-cx}A_{s}(c\phi )\right\}  
 \to  
 \phi _{t}-x+2\max _{0\le u\le s}\phi _{u}, && 
 \frac{1}{c}\log 
 \int _{s}^{t}e^{2c\phi _{u}}\,du
 \to 2\max _{s\le u\le t}\phi _{u},
\end{align*}
as well as $(1/c)\log A_{t}(c\phi )\to 2\max _{0\le u\le t}\phi _{u}$.
We combine these with \lref{;lfund} to get, as $c\to \infty $, 
\begin{align*}
 \ctx (\phi )(s)\to \phi _{s}-\max \Bigl\{ \phi _{t}-x+2\max _{0\le u\le s}\phi _{u},\,2\max _{s\le u\le t}\phi _{u}
 \Bigr\} +2\max _{0\le u\le t}\phi _{u}.
\end{align*}
The last expression is seen to coincide with $\cm _{x}(\phi )(s)$ because of the relations 
\begin{align*}
 2\max \{ a,b\} =a+b+|a-b|,\quad a,b\in \R , 
\end{align*}
and 
\begin{align*}
 \max _{0\le u\le t}\phi _{u}
 =\max \Bigl\{ 
 \max _{0\le u\le s}\phi _{u},\max _{s\le u\le t}\phi _{u}
 \Bigr\} .
\end{align*}
Consequently, \eqref{;ctxlq2} is proven and the proof of \lref{;ctxl} is completed.
\end{proof}

We are in a position to prove \lref{;mxpl}.

\begin{proof}[Proof of \lref{;mxpl}]
Pick a pair $(x,\phi )\in \R \times \ctd{t}{}$ fulfilling condition~\eqref{;mxbvc}. 
Since both sides of the claimed relation~\eqref{;mxplq1} are continuous at the 
origin, it suffices to show that 
\begin{align}\label{;mxplq1d}
 \cp (\cm _{x}(\phi ))(s)=\cp (\phi )(s)\quad \text{for all $0<s\le t$}.
\end{align}
Fix $s\in (0,t]$ below. By the definition of the transformation $Z$, we have 
the convergence 
\begin{equation}\label{;b}
 \begin{split}
  \frac{1}{c}\log Z_{s}(c\phi )&=\frac{1}{c}\log A_{s}(c\phi )-\phi _{s}\\
  &\xrightarrow[c\to \infty ]{}\cp (\phi )(s)
 \end{split}
\end{equation}
as to the right-hand side of \eqref{;ctxlq1}. To deal with its left-hand side, 
pick $\ve >0$ arbitrarily and take $c$ sufficiently large so that, by 
\eqref{;ctxlq2}, 
\begin{align*}
 \max _{0\le u\le t}\bigl| 
 \ctx (\phi )(u)-\cm _{x}(\phi )(u)
 \bigr| <\ve .
\end{align*}
Then, by the definition of $Z$, 
\begin{align*}
 e^{-2c\ve }e^{-c\ctx (\phi )(s)}A_{s}(c\cm _{x}(\phi ))
 \le Z_{s}\bigl( c\ctx (\phi )\bigr) \le 
 e^{2c\ve }e^{-c\ctx (\phi )(s)}A_{s}(c\cm _{x}(\phi )),
\end{align*}
from which it follows that 
\begin{align*}
 -2\ve +\cp (\cm _{x}(\phi ))(s)&\le 
 \liminf _{c\to \infty }\frac{1}{c}\log Z_{s}\bigl( c\ctx (\phi )\bigr) \\
 &\le \limsup _{c\to \infty }\frac{1}{c}\log Z_{s}\bigl( c\ctx (\phi )\bigr) 
 \le 2\ve +\cp (\cm _{x}(\phi ))(s).
\end{align*}
Thanks to the arbitrariness of $\ve $, we conclude that 
\begin{align*}
 \lim _{c\to \infty }\frac{1}{c}\log Z_{s}\bigl( c\ctx (\phi )\bigr) 
 =\cp (\cm _{x}(\phi ))(s).
\end{align*}
Combining this with \eqref{;b} and \eqref{;ctxlq1} leads to 
\eqref{;mxplq1d}, and hence to \eqref{;mxplq1}. The latter 
claim~\eqref{;mxplq2} is immediate from the former and 
\pref{;mxp}. 
\end{proof}

\subsection{Two fundamental lemmas in calculus}\label{;sstfl}

Given below are the statements and proofs of \lsref{;tdini} and \ref{;lfund} 
referred to in the proof of \lref{;ctxl}.

\begin{lem}\label{;tdini}
Let $I=[a,b]$ be a closed and bounded interval in $\R $ and $\{ f_{n}\} _{n=1}^{\infty }$ 
a sequence of real-valued nondecreasing functions on $I$ such that 
\begin{align*}
 \lim _{n\to \infty }f_{n}(s)=f(s)\quad \text{for all $s\in I$},
\end{align*}
with some continuous function $f$ on $I$. Then it holds that 
\begin{align*}
 \lim _{n\to \infty }\sup _{s\in I}\left| 
 f_{n}(s)-f(s)
 \right| =0.
\end{align*}
The same conclusion holds true if each $f_{n}$ is nonincreasing on $I$.
\end{lem}

\begin{proof}
We deal with the case that each $f_{n}$ is nondecreasing so that the 
limit function $f$ is also nondecreasing on $I$. Fix $\ve >0$ arbitrarily. 
By the uniform continuity of $f$, we may divide $I$ into a 
finite number of subintervals $[s_{k-1},s_{k}],\,k=1,\ldots ,N$, so that 
\begin{align}\label{;qdini1}
 f(s_{k})-f(s_{k-1})<\ve ,\quad k=1,\ldots ,N.
\end{align}
By the assumption of pointwise convergence, we take 
$n$ sufficiently large so that 
\begin{align*}
 \max _{0\le k\le N}\left| 
 f_{n}(s_{k})-f(s_{k})
 \right| <\ve .
\end{align*}
For every $s\in I$, pick a subinterval $[s_{k-1},s_{k}]$ containing $s$. 
Then we have, by the assumption that $f_{n}$ is nondecreasing, 
\begin{align*}
 f(s_{k-1})-\ve <f_{n}(s_{k-1})\le f_{n}(s)\le f_{n}(s_{k})<f(s_{k})+\ve .
\end{align*}
Together with 
\begin{align*}
 f(s_{k})<f(s)+\ve && \text{and} && f(s)-\ve <f(s_{k-1})
\end{align*}
by \eqref{;qdini1}, we have $\left| f_{n}(s)-f(s)\right| <2\ve $. 
\end{proof}

\begin{lem}\label{;lfund}
Given $\al ,\beta \in \R $, let $\{ a_{n}\} _{n=1}^{\infty }$, $\{ b_{n}\} _{n=1}^{\infty }$ and 
$\{ c_{n}\} _{n=1}^{\infty }$ be three sequences of positive real numbers 
such that, as $n\to \infty $, 
\begin{align*}
 \frac{1}{c_{n}}\log a_{n}\to \al , &&  
 \frac{1}{c_{n}}\log b_{n}\to \beta , && 
 c_{n}\to \infty .
\end{align*}
Then it holds that 
\begin{align*}
 \lim _{n\to \infty }\frac{1}{c_{n}}\log (a_{n}+b_{n})=\max\{ \al ,\beta \}. 
\end{align*}
\end{lem}

\begin{proof}
It is sufficient to consider the case $\al \ge \beta $. Fix $\ve >0$ arbitrarily. 
By assumption, there exists a positive integer $N$ such that, for all $n\ge N$, 
\begin{align*}
 e^{(\al -\ve )c_{n}}<a_{n}<e^{(\al +\ve )c_{n}}, 
 && e^{(\beta -\ve )c_{n}}<b_{n}<e^{(\beta +\ve )c_{n}},
\end{align*}
from which it follows that 
\begin{align*}
 \al -\ve +\frac{1}{c_{n}}\log \left\{ 
 1+e^{-(\al -\beta )c_{n}}
 \right\} &<\frac{1}{c_{n}}\log (a_{n}+b_{n})\\
 &<\al +\ve +\frac{1}{c_{n}}\log \left\{ 
 1+e^{-(\al -\beta )c_{n}}
 \right\} .
\end{align*}
Therefore, by the assumption that $c_{n}\to \infty $ as $n\to \infty $, 
\begin{align*}
 \al -\ve &\le \liminf _{n\to \infty }\frac{1}{c_{n}}\log (a_{n}+b_{n})\\
 &\le \limsup _{n\to \infty }\frac{1}{c_{n}}\log (a_{n}+b_{n})\le \al +\ve ,
\end{align*}
which proves the claim since $\ve >0$ is arbitrary. 
\end{proof}


\end{document}